\documentclass[reqno,12pt,a4paper]{amsart}
\usepackage{pgfplots}
\usepackage[english]{babel}
\usepackage{amsmath,amsthm,amssymb,amsfonts, mathdots}
\usepackage{scrtime, cancel}
\usepackage{enumitem, color}

\usepackage{cite}
\usepackage{mathtools}
\usepackage{ifpdf}
\ifpdf
\usepackage{hyperref}
\else
\usepackage[hypertex]{hyperref}
\fi
\usepackage{pdfsync,verbatim}
\usepackage{color}
\mathtoolsset{showonlyrefs}

\numberwithin{equation}{section}   

\allowdisplaybreaks
\newtheorem{theorem}{Theorem}[section]
\newtheorem{lemma}[theorem]{Lemma}
\newtheorem{proposition}[theorem]{Proposition}

\theoremstyle{definition}
\newtheorem{remark}[theorem]{Remark}

\definecolor{pinegreen}{rgb}{0.0, 0.47, 0.44}
\newcommand{\R}{\mathbb{R}}

\newcommand{\N}{\mathbb{N}}

\newcommand{\ellipses}{E}

\def\misgausskm{\gamma_{-\infty}}
\DeclareMathOperator{\tr}{tr}             
\newcount\dotcnt\newdimen\deltay
\def\Ddot#1#2(#3,#4,#5,#6){\deltay=#6\setbox1=\hbox to0pt{\smash{\dotcnt=1
\kern#3\loop\raise\dotcnt\deltay\hbox to0pt{\hss#2}\kern#5\ifnum\dotcnt<#1
\advance\dotcnt 1\repeat}\hss}\setbox2=\vtop{\box1}\ht2=#4\box2}

\def\Blue{\color{blue}}

\def\Green{\color{green}}

 \title[Maximal 
operator
  in an inverse Gaussian setting]{Boundedness properties of the  maximal 
operator\\
 in a nonsymmetric inverse Gaussian setting }
 \author[T.\ Bruno]{Tommaso Bruno}
\address{Dipartimento di Matematica, Universit\`a degli Studi di Genova\\Via Dodecaneso 35 \\16146 Genova\\ Italy}
\email{tommaso.bruno@unige.it}
\author[V.\ Casarino]{Valentina Casarino}
\address{DTG, Universit\`a degli Studi di Padova\\ Stradella san Nicola 3 \\I-36100 Vicenza \\ Italy}
\email{valentina.casarino@unipd.it}
\author[P.\ Ciatti]{Paolo Ciatti}
\address{Dipartimento di Matematica ``Tullio Levi Civita'', Universit\`a degli Studi di Padova\\Via Trieste, 63, 35131 Padova,  \\ Italy}
\email{paolo.ciatti@unipd.it}
\author[P.\ Sj\"ogren]{Peter Sj\"ogren}
\address{Mathematical Sciences,  University of Gothenburg and  Mathematical Sciences,
Chalmers University of Technology  \\ SE - 412 96 G\"oteborg, Sweden}
\email{peters@chalmers.se}

\thanks{The first  three authors are members of the Gruppo Nazionale per l'Analisi Matematica, la Probabilità e le loro Applicazioni (GNAMPA)
of the Istituto Nazionale di Alta Matematica (INdAM) and were partially supported by the INdAM--GNAMPA Project ``$L^{p}$ estimates for singular integrals in nondoubling settings'' (CUP E53C23001670001). 
This research was carried out while the fourth author was a visiting scientist at the University of Padova,  Italy, and he is grateful for its hospitality.}

\begin{document}
\begin{abstract}
We introduce a generalized inverse Gaussian setting and consider the maximal 
operator
 associated with the natural analogue of a nonsymmetric Ornstein--Uhlenbeck semigroup. We prove that it is bounded on $L^{p}$ when $p\in (1,\infty]$ and that it is of weak type $(1,1)$,
with respect to the relevant measure.
  For small values of the  time parameter $t$, the proof hinges on the ``forbidden zones" method previously introduced in the Gaussian context. But for large times the proof requires new tools.
\end{abstract}

\maketitle

\section{Introduction}\label{Introduction}
Alongside with the Gaussian framework, the Euclidean setting endowed with the absolutely continuous measure whose density is the reciprocal of a Gaussian has acquired independent interest in the last decade. It is nowadays known as the \emph{inverse Gaussian} setting. Its introduction in the realm of harmonic analysis dates back to 
F.~Salogni~\cite{Salogni}, who introduced the operator
 \[\mathcal A =\tfrac12\,\Delta +\langle x,\nabla \rangle\]
as an (essentially) self-adjoint operator in $L^2(\R^n, d\gamma_{-1})$. Here $d\gamma_{-1}$ stands for the inverse Gaussian measure
$$
d\gamma_{-1}(x) = \pi^{n/2}\, e^{ |x|^ 2}\,dx.
$$
In her PhD thesis~\cite{Salogni}, Salogni proved the weak type $(1,1)$ of the maximal operator associated to the semigroup  generated by $\mathcal A$. A few years later, T.~Bruno and P.~Sj\"ogren~\cite{Bruno2, BS} (see also~\cite{Bruno}) studied Riesz transforms and Hardy spaces for $\mathcal{A}$. Several other contributions to harmonic analysis in this context appeared  more recently, see~\cite{AB, ABFR, BMRM, BCLC, Mazzitelli}.

\medskip

While this development of the inverse Gaussian setting was going on,
 V.~Casarino, P.~Ciatti and P.~Sj\"ogren started  investigating        
several classical problems           
related to the semigroup generated by the generalized Ornstein--Uhlenbeck operator
\begin{equation}\label{generalizedOU}
\mathcal L^{Q,B}=
\tfrac12\,
\mathrm{tr}
\big( Q\nabla^2 \big)+\langle Bx, \nabla \rangle
\end{equation} 
in the Gaussian setting.
Here $\nabla$ indicates the gradient, $\nabla^2$ the Hessian, and $Q$ and  $B$ are two real  $n\times n$   matrices  called covariance and drift, respectively, satisfying
\begin{enumerate}
 \item[(H1)]
$Q$ is   symmetric and positive definite;
 \item[(H2)]
all the eigenvalues of  $B $  have negative real parts.
  \end{enumerate}
We refer the reader to
 \cite{CCS1, CCS2, CCS3, CCS4, CCS5, CCS6, CCS8} and the brief overview~\cite{CCS7}.
The semigroup generated by $\mathcal L^{Q,B}$ has the Gaussian probability measure
\begin{equation}\label{def:meas_gauss}
d\gamma_\infty (x)=
(2\pi)^{-\frac{n}{2}}
(\text{det} \, Q_\infty)^{-\frac{1}{2} }
\exp\big({-\tfrac12 \langle Q_\infty^{-1} x, x\rangle}\big)\,dx
\end{equation}
as invariant measure. 
Here
 $ Q_\infty$ is a certain symmetric and positive definite matrix whose precise definition in terms of $Q$ and  $B$ will be given in Section~\ref{s:Mehler}.

In some particular  cases, for instance  when   $Q$ and $-B$  coincide with the identity matrix $I_n$, 
the operator $\mathcal{L}^{Q,B}$ is essentially self-adjoint in $L^2(d\gamma_\infty)$,
 but in general it is not even symmetric. 
Thus many classical   problems, like the boundedness of singular integral operators associated with $\mathcal{L}^{Q,B}$, give rise to  interesting and nontrivial questions.
New techniques and new ideas are required, strong enough to overcome the lack of symmetry. 

\medskip

In this article, we combine the two approaches by considering a generalized version of the  inverse Gaussian setting studied in~\cite{Salogni, Bruno2, BS} via some of the techniques developed in~\cite{CCS1, CCS2, CCS3, CCS4, CCS5, CCS6, CCS8} and others which are new. This may be seen as the starting-point of a program whose goal is to develop an analogous theory in a nonsymmetric inverse Gaussian setting.

\smallskip

We shall replace the density of  $d\gamma_\infty$  by its reciprocal, i.e.,  $\R^n$ will be equipped with
 the inverse Gaussian measure
\[
d{\gamma_{-\infty}}(x)=
(2\pi)^{\frac{n}{2}}
(\text{det} \, Q_\infty)^{\frac{1}{2} }
\exp\big({\tfrac12 \langle Q_\infty^{-1} x, x\rangle}\big)\,dx  .
\]
Like $d\gamma_\infty$, this measure is obviously locally doubling but not globally doubling, but in contrast  to the probability measure $d\gamma_\infty$, it has superexponential growth at infinity.

In this setting, the  role which was played  by the Ornstein--Uhlenbeck operator in $(\R^n, d\gamma_\infty)$ is now played by the so-called inverse Ornstein--Uhlenbeck operator
\begin{equation}\label{def:opA}
\mathcal A^{Q,B}=
\tfrac12\,
\mathrm{tr}
\big( Q\nabla^2 \big)-\langle Bx, \nabla \rangle.
\end{equation}
 Here  $B$ and $Q$ are exactly the covariance and drift matrices
inherited  from the Ornstein--Uhlenbeck setting and satisfying (H1) and (H2). Notice that $d\gamma_{-1}$ and $\mathcal A$\,, considered  in~\cite{AB, ABFR, BMRM, BCLC, Bruno2, BS, Mazzitelli, Salogni},  are a special case
of $d\gamma_{-\infty}$ and $\mathcal A^{Q,B}$, corresponding to the choice $Q=-B=I_n$.

The semigroup generated by $\mathcal A^{Q,B}$   is      
$$
\mathcal H_t^{UO} := e^{t \mathcal A^{Q,B}}\,, \qquad t>0,
$$
(here ``$UO$'' has the scope of emphasizing the contrast with respect to ``$OU$'' which will be used for Ornstein--Uhlenbeck). While in the Gaussian context $d\gamma_\infty$ is invariant under the action of the Ornstein--Uhlenbeck semigroup, and  therefore chosen as the substitute for Lebesgue measure, $d\gamma_{-\infty}$ is not  invariant  under the action of $( \mathcal H_t^{UO})$. In fact, there are no invariant measures for $( \mathcal H_t^{UO})$; see, e.g.,~\cite{Metafune} or~\cite{Lunardi}. Nonetheless, $d\gamma_{-\infty}$ appears to be the natural measure for $( \mathcal H_t^{UO})$, because whenever $QB^{*}= BQ$ the operators $\mathcal H_t^{UO}$ are symmetric in $L^{2}(\R^n, d\gamma_{-\infty})$ (see Remark~\ref{remark:symmetry}). As we shall see in~Proposition~\ref{p:nucleoKt}, each operator of this semigroup is an integral operator, with a kernel  $K_t^{UO}$ with respect to $d\gamma_{-\infty}$.

\medskip

In this paper  we  study the boundedness of the maximal operator associated to the  semigroup $( \mathcal H_t^{UO})_{t>0}$, proving that it is of weak type $(1,1)$ and
of strong type $(p,p)$ for all $1<p\le \infty$, with respect to $d\gamma_{-\infty}$.
This extends a similar result in~\cite{Salogni}, proved under the assumption $Q=-B=I_n$.  Let us note that our results  appear to be the first of their kind for nonsymmetric operators  on manifolds with superexponential volume growth.

\smallskip

As is standard by now, the proof
 distinguishes between the  local and global parts of the  kernel $ K_t^{UO}$ 
of $\mathcal H_t^{UO}$. Here
 local and global mean   that  $|x-u| \leq 1/(1+|x|)$ and $|x-u| > 1/(1+|x|)$, respectively, $x,u$ being  the two 
arguments of the kernel.
Beyond this distinction
that dates back to~\cite{Mu, Peter, GC3},
the techniques in~\cite{Salogni}   rely on the spectral resolution of the self-adjoint operator
 $\mathcal A = \mathcal A^{I_n, -I_n}$ and seem  no longer   applicable in a nonsymmetric context.
Fortunately,    large parts of the machinery developed in~\cite{CCS1, CCS2, CCS3, CCS4, CCS5, CCS6, CCS8}  to study Gaussian harmonic analysis in a nonsymmetric setting
can be transferred to the inverse setting, and are useful to treat the local part of the maximal function and its global part for $t\in(0,1]$.
The global part for $t\ge 1$ is  more delicate and requires new tools.

\subsection*{Structure of the paper} In Section~\ref{s:Mehler}
we  recall some basic facts concerning the Ornstein--Uhlenbeck setting. We also compute
an explicit expression for the inverse Mehler kernel $ K_t^{UO}$ and  discuss its relationship with the Gaussian Mehler kernel.
In Section~\ref{simplifications}
the maximal operator associated to
$( \mathcal H_t^{UO})_{t>0}$ is introduced, and  the main theorem concerning its weak type $(1,1)$ and  strong type $(p,p)$ is stated.
We also give a theorem saying that for the global part and  $t\ge 1$ the weak type $(1,1)$ estimate can be enhanced
by  a logarithmic factor.
Section~\ref{preparations} contains
 some simplifications and reductions that prepare for  the proof of the  theorems, and Section~\ref{s:local}
is focused on
the local part of the maximal function. Then in Section~\ref{s:geometry}
 some  relevant geometric aspects of the problem are considered; in particular, we define
a system of polar-like coordinates used already in~\cite{CCS2}.
   Section~\ref{s:The global case}
concerns
 the global part of the maximal operator
for $0 < t \le 1$. In this case, the weak type $(1,1)$  is proved as a nontrivial  application of the ``forbidden zones",
  a recursive method introduced  first  by the fourth author in~\cite{Peter}.
The arguments for
 the global part with  $t > 1$ are given  in Section~\ref{s:mixed. t large}. Finally, in    Section~\ref{completion} we complete  the proofs
by putting together the various pieces. An argument showing that the enhanced result mentioned above is sharp ends the paper.


\subsection*{Notation}
 We shall denote by $C<\infty$ and $c>0$ constants that may vary from place to place. They  depend only on $n$,  $Q$ and  $B$, unless otherwise explicitly stated.
For two non-negative quantities $A$ and $B$, we write $A\lesssim B$, or equivalently  $B\gtrsim A$, if $A\leq C B$ for some $C$, and $A\simeq B$ means that $A\lesssim B$ and $B\lesssim A$.
By $\mathbb N$ we mean $\{0,1,\dots\}$. The symbol $T^*$ will denote the adjoint of  the operator $T$.

\section{The inverse Gaussian framework}\label{s:Mehler}

In this section we provide  explicit expressions for
the integral kernel of $\mathcal{H}_{t}^{UO}$ with respect  both to  Lebesgue measure and to $d\gamma_{-\infty}$ (see~\eqref{MtLeb} and Proposition~\ref{p:nucleoKt}, respectively).

In order to prove these formulae, we  need
some facts from the general Ornstein--Uhlenbeck setting.
Throughout the paper,   $B$ and $Q$  will be two real matrices
satisfying the hypotheses (H1) and (H2) introduced  in Section~\ref{Introduction}.

\subsection{Preliminaries}
We first recall the definition of  the covariance  matrices
\begin{equation}\label{defQt}
Q_t=\int_0^t e^{sB}Qe^{sB^*}ds, \qquad \text{ $t\in (0,+\infty]$}.
\end{equation}
Each $Q_t$ is well defined, symmetric  and positive definite.
Then we introduce the quadratic form
\begin{equation*}
R(x) ={\frac12 \left\langle Q_\infty^{-1}x ,x  \right\rangle}, \qquad\text{$x\in\R^n$.}
\end{equation*}
Sometimes, we shall  use the norm
\[ |x|_Q := | Q_\infty^{-1/2}\,x|,  \qquad x \in \R^n,\]
 which satisfies  $R(x)=\frac12|x|_Q^2$ and
 $|x|_Q \simeq |x|$.

 We also set
\begin{equation}\label{def:Dtx}
D_t = Q_\infty \,
 e^{-tB^*}\, Q_\infty^{-1},  \quad \; t \in \R,
\end{equation}
which is a one-parameter group of matrices.
In~\cite[Lemma 3.1]{CCS2} it has been  proved that
\begin{equation}\label{est:2-eBs-v}
  e^{ct}\,|x| \,\lesssim \,|D_t\, x| \, \lesssim  \, e^{Ct}\, |x|
  \qquad
\text{ and }
\qquad
  e^{-Ct}\,|x|\, \lesssim\, |D_{-t}\, x| \, \lesssim \,  e^{-ct}\, |x|,
\end{equation}
 for $t>0$ and all $x\in \R^n$.

When $x \ne 0$ and $0< t \le 1$,~\cite[Lemma 2.3]{CCS3} says that
\begin{equation}\label{x-Dtx}
 |x-D_t\, x|\simeq |t|\,|x|.
\end{equation}

\subsection{The inverse Mehler kernel}
The Ornstein--Uhlenbeck operator $\mathcal L^{Q,B}$ given by~\eqref{generalizedOU} is essentially selfadjoint in $L^2(\gamma_\infty)$; the measure $d\gamma_\infty$ is defined in~\eqref{def:meas_gauss}. We will sometimes write $\gamma_\infty(x)$ for its density.
For each $f\in L^1(\gamma_\infty)$ and all $t>0$ one has
\begin{align}                       \label{def-int-ker}
e^{t \mathcal L^{Q,B}} f(x) &=
 \int
K_t^{OU}
(x,u)\,
f(u)\,
 d\gamma_\infty(u)
  \,, \qquad x\in\R^n,
\end{align}
 where for $x,u\in\R^n$ and $t>0$ the Mehler kernel $K_t^{OU}$ (with respect to $d\gamma_\infty$)   is given by
\begin{align}    \label{mehler}                   
K_t^{OU} (x,u)\!
=\!
\Big(
\frac{\det \, Q_\infty}{\det \, Q_t}
\Big)^{\frac12 }
e^{R(x)}\,
\exp \bigg[
{-\frac12
\left\langle \left(
Q_t^{-1}-Q_\infty^{-1}\right) (u-D_t \,x) \,, u-D_t\, x\right\rangle}\bigg];\,\;\;\;
\end{align}
see~\cite[(2.6)]{CCS2}.
This immediately yields
\begin{align}                       \label{def-int-ker_Leb}
e^{t \mathcal L^{Q,B}}
 f(x) &=
 \int
M_t^{OU}
(x,u)\,
f(u)\,
 du
  \,, \qquad x\in\R^n,
\end{align}
where the  kernel
$M_t^{OU}$ (with respect to  Lebesgue measure)  fulfills
\begin{align}\label{defMtQB}
\mathcal  \mathcal M_t^{OU}(x,u)= &\,K_t^{OU} (x,u)\,\gamma_\infty(u)\\
= & \,(2\pi)^{-\frac n 2} (\det Q_t)^{-\frac12}\,  e^{R(x)-R(u)}
 \exp\left[ -\frac12
  \langle (Q_t^{-1} -Q_\infty^{-1})(u-D_t\, x), \, u-D_t\, x \rangle    \right].
\end{align}
From this we can deduce the corresponding kernel for the inverse  Ornstein--Uhlenbeck setting.

\begin{lemma}\label{l:kernel}
 The kernel
 of $e^{t \mathcal A^{Q,B}}$ with respect to  Lebesgue measure  is
\begin{equation*}                  
M_t^{UO}(x,u)=
 e^{t\,\mathrm{tr} B}\, M_t^{OU} (u,x) ,
\end{equation*}
for any $x,u\in\R^n$ and $t>0$.
\end{lemma}
\begin{proof}
We first compute the adjoint of $\mathcal L^{Q,B}$ in $L^2(\mathbb R^n, dx)$, where $dx$ denotes  Lebesgue measure.
 Let $f$ and $g$ be smooth functions with compact supports in $\mathbb R^n$.
 The second-order term in  $\mathcal L^{Q,B}$
is symmetric, and for the first-order term we integrate by parts, getting
\begin{multline*}
\langle f, \mathcal L^{Q,B}\,g\rangle =   \int f(x)\, \left(\frac12\,\mathrm{tr} (Q\nabla^2g)(x) +
\langle Bx, \nabla g(x)\rangle\right)\,dx \\
  =  \int  \left(\frac12\,\mathrm{tr} (Q\nabla^2f)(x) - \langle Bx, \nabla f(x)\rangle - \mathrm{tr} B   f(x)\right)
  g(x)\,dx   =  \,\langle \mathcal A^{Q,B}f - \mathrm{tr} B f, g\rangle.
\end{multline*}
Thus \[ \mathcal A^{Q,B} = \big(\mathcal{L}^{Q,B}\big)^{*} +  \mathrm{tr} B,\] 
and
\[
e^{t \mathcal A^{Q,B}} = e^{t\,\mathrm{tr} B} \,e^{t(\mathcal L^{Q,B})^{*}}, \qquad t>0.
\]
Since $e^{t(\mathcal{L}^{Q,B})^{*} }$ is the adjoint of  $e^{t\mathcal{L}^{Q,B}}$, it has kernel
 $ (M_t^{OU})^* (x,u) = M_t^{OU} (u,x)$, whence the claim.
 \end{proof}


From~\eqref{defMtQB} and Lemma~\ref{l:kernel}  we have
\begin{align} \label{MtLeb}
M_t^{UO}(x,u)&\\
  = (2\pi)^{-\frac n2}& (\det Q_t)^{-\frac12}\,  e^{t\,\mathrm{tr} B}\,  e^{R(u)-R(x)} \notag
\, \exp\left[ -\frac12
  \langle (Q_t^{-1} -Q_\infty^{-1})(x-D_t\, u), \, x-D_t \,u \rangle    \right].
\end{align}
This  kernel is for integration against  Lebesgue measure, but  the relevant measure in the inverse setting is $d\gamma_{-\infty}$.
Dividing
$M_t^{UO}$
 by the density $\gamma_{-\infty}(u)$,
one obtains the kernel of  $e^{t \mathcal A^{Q,B}}$ for integration against    $d\gamma_{-\infty}$, as follows.
\begin{proposition}\label{p:nucleoKt}
 The kernel
 of $e^{t \mathcal A^{Q,B}}$ with respect to $d\gamma_{-\infty}$   
is
\begin{multline*}    
K_t^{UO}(x,u)   \\  =
(2\pi)^{-n}
\big(
 \det Q_\infty\, \det Q_t\big)^{-\frac12}\,  e^{t\,\mathrm{tr} B}\,  e^{-R(x)}
   \, \exp\left[ -\frac12
  \langle (Q_t^{-1} -Q_\infty^{-1})(x-D_t\, u), \,x-D_t\, u \rangle    \right]
  \end{multline*}
for all   $x,u\in\R^n$,
 and $t>0$.
\end{proposition}

It follows that the Mehler kernel $ K_t^{OU}$
and its counterpart $ K_t^{UO}$  in the inverse Gaussian setting are related by
\begin{equation}\label{duenuclei}
K_t^{UO}(x,u) =
(2\pi)^{-n}
(\det Q_\infty)^{-1}\,    e^{-R(x)} \, e^{-R(u)}\, e^{t\,\mathrm{tr} B}\,K_t^{OU}(u,x).
\end{equation}

\begin{remark}\label{remark:symmetry}
Because of \eqref{duenuclei}, the semigroup $e^{-t\mathcal{A}^{Q,B}}$ is symmetric on $L^{2}(\gamma_{-\infty})$ if and only if $e^{-t\mathcal{L}^{Q,B}}$ is symmetric on $L^{2}(\gamma_{\infty})$. Since the latter property holds if and only if $QB^* = BQ$ by~\cite[Theorem 2.4]{CMG} (see also \cite[Lemma 2.1]{MPRS}), we see that also $e^{-t\mathcal{A}^{Q,B}}$  is symmetric on $L^{2}(\gamma_{-\infty})$ if and only if $QB^* = BQ$.
\end{remark}

Formula~\eqref{duenuclei} allows us to transfer the upper and lower estimates of  $  K_t^{OU}$ in~\cite[formulae~(3.4) and (3.5)]{CCS2}
to  $K_t^{UO}$, as follows.
\begin{proposition}\label{esr_kernelQB}
 If $0 < t\le 1$, one has
\begin{equation}\label{litet}
 \frac{ e^{-R( x)}}{t^{n/2}}  \exp\bigg(-C\,\frac{|D_{-t}\,x- u |^2}t\bigg)
 \lesssim   K_t^{UO}(x,u)
 \lesssim  \frac{ e^{-R( x)}}{t^{n/2}} \exp\bigg(-c\,\frac{|D_{-t}\,x- u |^2}t\bigg)
\end{equation}
for all $(x, u) \in\mathbb R^n \times \mathbb R^n$. If instead $t\ge 1$, then
\begin{equation}\label{stort}
e^{-| {\tr} B|t}\,  e^{-R(x)}
\exp\left( -C
 | D_{-t}\,x- u|^2    \right)\lesssim
K_t^{UO}(x,u) \lesssim
 e^{-| {\tr} B| \,t}\,  e^{-R(x)}
\exp\left( -c
 | D_{-t}\,x- u|^2    \right)
\end{equation}
for all $(x,u)\in\mathbb R^n \times \mathbb R^n$.
\end{proposition}

 \subsection{A Kolmogorov-type formula}
We conclude this section by deducing a Kolmogorov-type formula in the inverse Gaussian setting. Recall that on the space $\mathcal C_b(\R^n)$
of bounded continuous functions,
 the
Ornstein--Uhlenbeck semigroup
$(\mathcal H_t^{OU})_{t> 0}$ is explicitly given
 by the Kolmogorov formula
 \begin{equation}\label{Kolmo}
 \mathcal H_t^{OU}
f(x)=(2\pi)^{-\frac{n}{2}}
({\det} \, Q_t)^{-\frac{1}{2} }
\int
f(e^{tB}x-u)\, e^{-\frac12 \langle Q_t^{-1}u,u\rangle}\,d\gamma_t (u)\,, \quad x\in\R^n\,,
\end{equation}
see e.g.~\cite{Kolmo}.  Here  $d\gamma_t$  denotes a normalized Gaussian measure in $\R^n$ given by
\[d\gamma_t (x)= (2\pi)^{-\frac{n}{2}} (\text{det} \, Q_t)^{-\frac{1}{2} } e^{-\frac12 \langle Q_t^{-1}x,x\rangle}dx  \,,\qquad \text{ $t\in (0,+\infty]$. }\]
\begin{proposition}\label{p:Kolmo}
 For all $f\in \mathcal C_b(\R^n)$
 one has
\begin{equation*}
\mathcal H_t^{UO}
f(x)=(2\pi)^{-\frac{n}{2}}
({\det} \, Q_t)^{-\frac{1}{2} }
\int
f(e^{-tB}(u+x))  \, e^{-\frac12 \langle Q_t^{-1}u,u\rangle}\, du\,, \qquad x\in\R^n\,.
\end{equation*}
\end{proposition}
\begin{proof}
From~\eqref{Kolmo}
it follows that
\begin{align*}
\mathcal H_t^{OU}
f(x)&=(2\pi)^{-\frac{n}{2}}
(\text{det} \, Q_t)^{-\frac{1}{2} }
\int
f(u)\,e^{-\frac12 \langle Q_t^{-1}(e^{tB}x-u),(e^{tB}x-u)\rangle}
du \,, \qquad x\in\R^n\,,
\end{align*}
  which means that
\begin{equation*}                
M_t^{OU}
(x,u)\,
=(2\pi)^{-\frac{n}{2}}
(\text{det} \, Q_t)^{-\frac{1}{2} }
\,e^{-\frac12 \langle Q_t^{-1}(e^{tB}x-u),(e^{tB}x-u)\rangle}
\,, \qquad \quad x,u \in\R^n\,.
\end{equation*}
This and Lemma~\ref{l:kernel}  imply that
\begin{align*}
\mathcal H_t^{UO}
f(x)
&= e^{t\,\mathrm{tr} B}\, \int f(u) \,M_t^{OU} (u,x)\, du
\\
&= e^{t\,\mathrm{tr} B}\, (2\pi)^{-\frac{n}{2}}
(\text{det} \, Q_t)^{-\frac{1}{2} }
\int f(u) \,
\,e^{-\frac12 \langle Q_t^{-1}(e^{tB}u-x),(e^{tB}u-x)\rangle}\, du
\\
&=(2\pi)^{-\frac{n}{2}} ({\det} \, Q_t)^{-\frac{1}{2} }
\int
f(e^{-tB}(u+x)) \, e^{-\frac12 \langle Q_t^{-1}u,u\rangle} du\,, \qquad x\in\R^n,
\end{align*}
where the last equality follows by a trivial change of variables. This proves the assertion.
\end{proof}
\begin{remark}\label{conservative}
Proposition~\ref{p:Kolmo} immediately implies that $(\mathcal H_t^{UO})_{t>o}$ is conservative,  that is, each $\mathcal H_t^{UO}$ maps the constant
function $\mathbf{1}$ into itself.
\end{remark}

\section{The main result}\label{simplifications}
The  maximal operator associated to the  semigroup $\big( \mathcal H_t^{UO}\big)_{t>0}$ is defined as
 \begin{align}\label{OU-max}
 \mathcal H_*
f(x)=
\sup_{t> 0}
\big|
 \mathcal H_t^{UO}
f(x)
\big|.
\end{align}
Notice that we omit 
indicating
 that  $\mathcal H_*$ refers to the inverse Ornstein--Uhlenbeck semigroup. Our main result is the following.
\begin{theorem}\label{weaktype1}
The inverse Ornstein--Uhlenbeck maximal operator
$\mathcal H_*$
 is bounded from $L^1(\gamma_{-\infty})$
 to $L^{1,\infty}(\gamma_{-\infty})$, and on $L^p(\gamma_{-\infty})$
  for all $1<p\le \infty$. 
 \end{theorem}
 We first deal with the strong type $(p,p)$.  
 Remark~\ref{conservative} implies that $\mathcal H_*$ is bounded on $L^\infty(\gamma_{-\infty})$. Given the weak type $(1,1)$, one can then interpolate to obtain the boundedness on $L^p(\gamma_{-\infty})$ for  $1<p< \infty$.

 All we have to do is thus to prove the  weak type $(1,1)$ of $\mathcal H_*$, that is,  the estimate
\begin{equation} \label{thesis-mixed-Di}
\gamma_{-\infty}
\{x\in\R^n \colon \mathcal H_*f(x) > \alpha\} \le \frac{C}\alpha\,\|f\|_{L^1( \gamma_{-\infty})},\qquad \text{ $\alpha>0$,}
\end{equation}
 for all functions $f\in L^1 (\gamma_{-\infty})$ and some  $C=C(n,Q,B) < \infty$. Let us emphasize that we do not keep track of the
 precise
  dependence of the constants on the parameters $n$, $Q$ and $B$.

If we consider only the supremum over $t\ge 1$, the estimate~\eqref{thesis-mixed-Di} can be improved for small $\alpha$, as follows.
\begin{theorem}\label{c:sharp}
Suppose  $f\in L^{1}(\gamma_{-\infty})$ has norm $1$, and take $\alpha \in (0, 1/2)$. Then
\begin{equation} \label{ineq:enhanced}
\gamma_{-\infty}
\left\{x\in\R^n \colon \sup_{t>1}\big| \mathcal H_t^{UO} f(x) \big|> \alpha\right\} \lesssim \frac{1}{\alpha\,\sqrt{\log( 1/\alpha)} }.                                
\end{equation}
This estimate is sharp in the following sense. If           
\begin{equation}            \label{enhanced}
\gamma_{-\infty}
\left\{x\in\R^n \colon \sup_{t >1}\big| \mathcal H_t^{UO} f(x) \big|> \alpha\right\} \lesssim \frac{1}{\Phi(\alpha)},                                \qquad 0 < \alpha < \frac12
\end{equation}
where $\Phi$ is a function defined in  $(0, 1/2)$, then 
\begin{equation}   \label{PHI}
\Phi(\alpha) = \mathcal O\left(\alpha\,\sqrt{\log( 1/\alpha)}\,\right) \qquad \mathrm{as}  \qquad \alpha \to 0.
\end{equation}
 \end{theorem}

A similar improvement has been observed in the Ornstein-Uhlenbeck setting, see~\cite{CCS2}.

\section{Preparation for the proof}\label{preparations}
We  introduce
some simplifications and reductions which will be useful  in the proofs of
the theorems,                    
 and do away with some simple cases.

First of all, we may assume
that
$f$ is nonnegative
and normalized in $L^1( \misgausskm)$.

\medskip

\subsection{Splitting of the operator}\label{splitting}
We consider separately the supremum in~\eqref{OU-max} for small and large values of $t$. Further, we
 divide the operator into a local and a global part, by means of
 a local region
 \begin{align*}
L&=\left\{
(x,u)\in\R^n\times\R^n\,:  \,|x-u|\le \frac{1}{1+|x|}
\right\}
\end{align*}
and a global region
 \begin{align*}
G&=\left\{
(x,u)\in\R^n\times\R^n\,:  \,|x-u|> \frac{1}{1+|x|}
\right\}.
\end{align*}
Clearly,  $\mathcal H_*$ is dominated by the sum of the following four operators
\begin{align*}
 \mathcal H_*^{-,L}
f(x)=
\sup_{t\le 1}
\Big|
 \int  K_t^{UO}(x,u)\,\mathbf{1}_L(x,u)\,
f(u)\, d\gamma_{-\infty}(u)
\Big|;
\end{align*}
\begin{align*}
 \mathcal H_*^{-,G}
f(x)=
\sup_{t\le 1}
\Big|
 \int   K_t^{UO}(x,u)\,\mathbf{1}_G(x,u)\,
f(u)\, d\gamma_{-\infty}(u)
\Big|;
\end{align*}
\begin{align*}
 \mathcal H_*^{+,L}
f(x)=
\sup_{t> 1}
\Big|
 \int   K_t^{UO}(x,u)\,\mathbf{1}_L(x,u)\,
f(u)\, d\gamma_{-\infty}(u)
\Big|;
\end{align*}
\begin{align*}
 \mathcal H_*^{+,G}
f(x)=
\sup_{t> 1}
\Big|
 \int  K_t^{UO}(x,u)\,\mathbf{1}_G(x,u)\,
f(u)\, d\gamma_{-\infty}(u)
\Big|.
\end{align*}

We will prove   estimates like~\eqref{thesis-mixed-Di}  for each of these four operators
and~\eqref{ineq:enhanced} for the latter two.

\subsection{Simple upper bounds}\label{simpleupper}
\begin{lemma}\label{simple}
  If $f\ge 0$ is normalized in  $L^1( \misgausskm)$, then
  \begin{equation*}
    \mathcal H_*^{+,L}f + \mathcal H_*^{+,G}f + \mathcal H_*^{-,G}f \lesssim 1.
  \end{equation*}
\end{lemma}

 \begin{proof}
 If $t>1$, we see from~\eqref{stort} that $K_t^{UO}(x,u) \lesssim 1$ for all $(x,u)$.
This implies the claimed estimate for
 $\mathcal H_*^{+,L}$ and $\mathcal H_*^{+,G}$. To deal with  $\mathcal H_*^{-,G}$, we need another lemma.
\begin{lemma}\label{lemma-3.5}
  If $(x,u)\in  G$ and $0<t\le 1$, then
 \begin{equation*}
K_t^{UO}(x,u)
\lesssim   e^{-R(x)}\, (1+|x|)^n.                    
\end{equation*}
 \end{lemma}
 \begin{proof}
   We apply the definition of $G$ and then~\eqref{x-Dtx}, to get
\begin{equation}\label{forglobal}    
\frac1{1+|x|} <  |u-x| \le |u - D_{-t}\,x| + |D_{-t}\,x -x|  \le |u - D_{-t}\,x| + Ct|x|.
\end{equation}
Thus $ |D_{-t}\,x -u|\ge \frac 1{1+|x|} - C t|x|$, and
\begin{equation*}
\frac{|D_{-t}\,x -u|^2}t \ge \frac 1{t(1+|x|)^2} - C\frac{|x|}{1+|x|} \ge \frac 1{t(1+|x|)^2} - C.
\end{equation*}
From~\eqref{litet} we then see that
\begin{equation*}
K_t^{UO}(x,u) \lesssim  e^{-R(x)}\,t^{-n/2}\, \exp\left(- \frac c{t(1+|x|)^2}\right)
 \lesssim  e^{-R(x)}\, (1+|x|)^n,
\end{equation*}
and Lemma~\ref{lemma-3.5} is proved.
 \end{proof}

To complete the proof of   Lemma~\ref{simple}, it is now enough to observe that  Lemma~\ref{lemma-3.5}  implies  $K_t^{UO}(x,u) \lesssim 1$ for $(x,u) \in G$ and $0 < t \le 1$, and thus
 $\mathcal H_*^{-,G}f \lesssim 1$.
\end{proof}

The following consequence of Lemma~\ref{simple} will be useful.

Let $\alpha_0 = \alpha_0(n,Q,B)\in (0,1/2]$. To prove the estimates~\eqref{thesis-mixed-Di} and~\eqref{ineq:enhanced}
for any of the three operators in Lemma~\ref{simple}, it is enough to estimate the relevant level set for levels $\alpha < \alpha_0$, with $f$ normalized in $L^1( \misgausskm)$. Indeed, Lemma~\ref{simple} says that the level set is empty for levels larger than some $C$. For levels $\alpha \in  (\alpha_0, C]$, one can use the estimate corresponding to the level $\alpha_0/2$, since then both sides of the inequalities~\eqref{thesis-mixed-Di} and~\eqref{ineq:enhanced} will be of order of magnitude~1.

\subsection{Observations for small levels $\alpha$}\label{smalllevels}
Assuming $\alpha < 1/2$, we can estimate the set where $R(x)$ is not too large. Indeed,
  \begin{align*}
\gamma_{-\infty}
\left\{
x\in
\R^n \colon R( x) <   \frac{3}{4}  \log\frac1\alpha\,
\right\}
&=
\int_{ R( x)\le \frac{3}{4}  \log\frac1\alpha}  e^{R(x)}\,dx\,,
 \end{align*}
 and in the integral here we change variables to $x' = Q_\infty^{-1/2}x$ so that $R(x) = |x'|^2/2$.
 Then one passes to polar coordinates and finds that
 \begin{align} \label{restrizione-prima-1}
\gamma_{-\infty}
 \left\{
x\in
\R^n
\colon
\,R( x) <   \frac{3}{4}  \log\frac1\alpha\,
\right\}
&\simeq
\frac1{\alpha^{3/4}}
\,
\Big(\log \frac1\alpha\Big)^{(n-2)/2}\,
\lesssim  \frac{1}{\alpha\,\sqrt{\log( 1/\alpha)} }.
 \end{align}
The last, simple estimate here is stated in view of Theorem~\ref{c:sharp}.
As soon as  $\alpha < 1/2$, we can thus neglect the set where $R( x) <   \frac{3}{4}  \log(1/\alpha)$
when we prove   Theorems~\ref{weaktype1} and~\ref{c:sharp}.

Further, if  $(x,u)\in  G$ and $R(x) > \frac54\,\log(1/\alpha)$ with   $\alpha < 1/2$, then Lemma~\ref{lemma-3.5} implies for $t \le 1$
\begin{equation}\label{5/4}
K_t^{UO}(x,u) \lesssim  \alpha^{5/4} \,\left(\log\frac1\alpha\right)^{n/2}
\lesssim  \alpha.
\end{equation}
When  $t > 1$, this remains true, since~\eqref{stort} then shows that $K_t^{UO}(x,u) \lesssim e^{-R(x)}$.

For the operators  $\mathcal H_*^{\pm,G}$ we need thus only consider points $x$ in the annulus
\begin{equation}\label{def:annulus}
\mathcal E_\alpha = \left\{x\in \R^{n} \colon \frac34 \log \frac{1}{\alpha} \le R(x) \le \frac54 \log \frac{1}{\alpha}\right\}.
   \end{equation}

\section{The local part of the maximal function}\label{s:local}

This section consists of the proof of the following result.

\begin{proposition}\label{prop-locale}
The  operators  $ \mathcal H_*^{+,L}$ and  $ \mathcal H_*^{-,L}$
  are of weak type $(1,1)$   with respect to the measure $d\misgausskm$.
  Further,  $ \mathcal H_*^{+,L}$  satisfies also the sharpened estimate of Theorem~\ref{c:sharp}.
 \end{proposition}


  We start with    $H_*^{+,L}$, which is of strong type (1,1) with respect to  $d\misgausskm$.
  Indeed,  \eqref{stort}   shows that $K_t^{UO}(x,u) \lesssim e^{-R(x)}$ for $t>1$.
 With      $0 \le f \in L^1(\misgausskm)$ we then have
    \begin{equation*}
      H_*^{+,L}\,f(x) \lesssim e^{-R(x)}\,\int \mathbf{1}_{L}(x,u)\,f(u)\,d\misgausskm(u).
    \end{equation*} 
    Hence,
     \begin{align} \label{ny}
  \int     H_*^{+,L}\,f(x)\,d\misgausskm(x) & \lesssim      
  \int f(u) \int \mathbf{1}_{L}(x,u)\, dx\,d\misgausskm(u),
    \end{align}  
    where we swapped the order of integration and passed to Lebesgue measure $dx$.
    Now $(x,u) \in L$ implies $1+|x| \simeq 1+|u|$ and thus $|x-u| \lesssim 1/(1+|u|)$, so the
     inner integral here is no larger than 
    $C(1+|u|)^{-n} \lesssim 1$. The strong type $(1,1)$ follows. 
    To obtain also  \eqref{ineq:enhanced} for    $H_*^{+,L}$,
    we take  $\alpha \in (0,1/2)$ and see from \eqref{restrizione-prima-1} that  we need only consider points $x$ with $|x| \gtrsim \sqrt{\log (1/\alpha)}$. When restricted to such $x$, the above inner integral is no larger than
     $C\left(\log (1/\alpha)\right)^{-n/2}$.
     Chebyshev's inequality then implies that  $H_*^{+,L}$ satisfies \eqref{ineq:enhanced}.

   When we now deal with  $H_*^{-,L}$,    we can
     replace  $d\gamma_{-\infty}$ by Lebesgue measure.             
     We briefly describe the argument for this, which is  based on a covering 
     procedure  to be found in~\cite[Subsection 7.1]{CCS5} and for the standard Gaussian measure also in~\cite[Lemma 3.2.3]{Salogni}. Notice that in~\cite{CCS5} the localization is by means of balls, but the same method works with cubes instead of balls.     The idea is to cover $\R^n$ by pairwise disjoint cubes $Q_j$ with centers  $c_j$ and of sides roughly $1/(1+|c_j|)$. Then $f = \sum f\mathbf{1}_{ Q_j}$, and $ \mathcal H_*^{-,L} (f\mathbf{1}_{ Q_j})$
is supported in a concentrically scaled cube  $ CQ_j$. These scaled cubes have bounded overlap, and in each   $ CQ_j$ the measure  $d\gamma_{-\infty}$ is essentially proportional to Lebesgue measure.
It is therefore enough to verify the weak type $(1,1)$ of  $ \mathcal H_*^{-,L}$ with respect to Lebesgue measure in  $\R^n$, and then apply this to 
  each  $f\mathbf{1}_{Q_j}$ and  sum in $j.$

Observe that  $(x,u)\in L$ implies
\begin{align}\label{RuRxinL}
|R(u)-R(x)|&= \frac{1}{2}\,\big||u|_{Q}^{2} - |x|^{2}_{Q}\big|
= \frac{1}{2}\, \big| |u|_{Q}-|x|_{Q}\big|\left(|u|_{Q}+|x|_{Q}\right) \\
&\lesssim |u-x|_Q \,\left(|u-x|_Q+ 2\,|x|_Q\right) \lesssim \frac{1}{(1+|x|)^2}+ \frac{|x|}{1+|x|} \lesssim 1.
\end{align}

 If $(x,u) \in L$
and $0<t\le 1$ we have    because of~\eqref{x-Dtx}                                     
\begin{align*}
   |D_{-t}\, x-u|  =   |D_{-t}\, x -x+x-u| \ge
|u-x| - |D_{-t}\, x -x| \gtrsim |u-x| - Ct|x|.
\end{align*}
Squaring and dividing by $t$,  we get
\begin{equation}                              
\frac{|D_{-t}\, x-u|^2}t \gtrsim \,\frac{|u-x|^2}t -C\,|x-u||x| \ge  \,\frac{|u-x|^2}t -C.
\end{equation}
In view of~\eqref{litet}, this implies that the relevant kernel        
is no larger than constant times
  \begin{equation*}
e^{-R(x)}\,{  t^{-n/2}
}
\exp{
 \Big(-c\,
 \frac{|
 u-x|^2}{ t} \,
\Big)
}\,\mathbf{1}_{L}(x,u)
\,.
\end{equation*}
With $f \ge 0$ then
\begin{align*}
\mathcal H_*^{-,L}
f(x)
&\lesssim
\sup_{0<t\le  1}
\frac{e^{-R(x)}}{
t^{n/2}}
 \int
\exp
 \Big(-c\,
 \frac{|
 x-u|^2
 }{ t} \,
\Big)
\,\mathbf{1}_{L}(x,u)\,f(u)
\,
 d\gamma_{-\infty}(u)\\
&\simeq
\sup_{t>0}
\frac{1}{
t^{n/2}}
 \int
\exp
 \Big(-c\,
 \frac{|
 x-u|^2
 }{ t} \,
\Big)
\,\mathbf{1}_{L}(x,u)\,f(u)
\,
 du,
\end{align*}
where the second step relied on~\eqref{RuRxinL}.
The last supremum here defines a maximal operator which is  well known to be dominated by the Hardy--Littlewood maximal function and so  of
weak type $(1,1)$                
 with respect to  Lebesgue measure in $\R^n$, see e.g.~\cite[p.\ 73]{Stein}. 
 
 Proposition~\ref{prop-locale} is proved.

\vskip11pt

\section{Some geometric background}\label{s:geometry}

\subsection{Polar coordinates}
For $\beta > 0$
we introduce  the ellipsoid
\begin{equation*}
\ellipses_\beta
=\{x\in\R^n\colon R(x)=\beta
\}
\,.\end{equation*}
Then we recall
\cite[formula (4.3)]{CCS2} saying that
\begin{align}
&\frac{\partial}{\partial s}\,
R\big( D_s\,x \big)
 =\frac12\,
\big|
Q^{1/2} \,  Q_\infty^{-1}   D_s \,x\big|^2
\simeq
\big|
D_s \,x\big|^2
 \label{vel-4}
\end{align}
for all $x$
in $\R^n$ and $s\in\R$.
If
 $x\neq 0$ the map
$s\mapsto R(D_s \,x)$ is thus strictly increasing, and $x$
can be written uniquely as 
\begin{equation}\label{def-coord}
x=D_s\, \tilde x
\end{equation}
for some $\tilde x\in \ellipses_\beta$
and $s\in\R$. We call $s$ and $\tilde x$ the
 polar coordinates of $x$.                  

From~\cite[Proposition 4.2]{CCS2} we recall that
the Lebesgue measure in $\Bbb R^n$ is given in terms of polar coordinates
$(s, \tilde x)$ by
\begin{align}\label{def:leb-meas-pulita}
  dx =
e^{-s\tr B}\, \frac{ |Q^{1/2}\, Q_\infty^{-1} \tilde x |^2}
{2\,| Q_\infty^{-1} \tilde x  |}\,
 dS( \tilde x)\,ds\,,                        
\end{align}
where $dS$ denotes the area measure of $E_\beta$.

\medskip
We shall use   the estimates for the distance between two points expressed in terms  of  polar  coordinates given in~\cite[Lemma 4.3]{CCS2}.
Let $x^{(0)},\; x^{(1)}\in \R^n
 \setminus \{ 0\} $ be given by
  $$
  x^{(0)} = D_{ s_0}\,
\tilde x^{(0)}
\qquad
 \text{
and   }
\qquad x^{(1)} = D_{s_1} \,\tilde x^{(1)},
$$
 with  $s_0$,  $s_1 \in \R$ and
$\tilde x^{(0)},\; \tilde x^{(1) }\in E_\beta$.
If  $ R(x^{(0)}) > \beta/2$, then
\begin{equation}
  \label{lem1}
  \big|x^{(0)} - x^{(1)}\big| \gtrsim
  \big|\tilde x^{(0)} - \tilde x^{(1)}\big|.
  \end{equation}
 We observe here that the assumption  $ R(x^{(0)}) > \beta/2$ may be weakened to  $ R(x^{(0)}) > c\, \beta$ with $c = c(n,Q,B)>0$, as can be seen from the proof in~\cite{CCS2}.

If   $ R(x^{(0)}) > \beta/2$ and also $s_1 \ge 0$, the same lemma says that
\begin{equation}
  \label{lem2}
  \big|x^{(0)} - x^{(1)}\big| \gtrsim \sqrt\beta\,|s_0 -s_1|.
\end{equation}

\medskip

\subsection{
The inverse Gaussian measure of a tube}\label{restriction}
We fix a large  $\beta > 0$.
Define for $y\in E_\beta$ and $a>0$ the set
\begin{equation*}
  \Omega = \left\{
  x \in  E_\beta\colon \left|x - y\right| < a \right\}.
\end{equation*}
This is  a spherical cap of the ellipsoid $E_\beta$,
centered at $y$.
The area of
 $\Omega$ is $S(\Omega)\simeq \, a^{n-1}\wedge \beta^{(n-1)/2} $. Then
consider the tube
\begin{equation}   \label{zona}
Z =  \{D_{s}\, \tilde x
\colon s\leq 0,
\;\tilde x \in \Omega \}.
\end{equation}

\begin{lemma}\label{lemma-Peter-forbidden}
There exists a constant $C$ such that for $\beta > C$
the inverse Gaussian
measure
 of the tube $Z$ fulfills
\begin{equation*}\label{stima-Peter-mu}
\gamma_{-\infty} (Z)\lesssim
\frac{a^{n-1}}{\sqrt{ \beta}}\, e^{\beta}.
\end{equation*}
\end{lemma}

\begin{proof}
From~\eqref{def:leb-meas-pulita}
we obtain, since $ |\tilde x|\simeq \sqrt \beta$ and $R(\tilde x) = \beta$,
\begin{align*}   
\gamma_{-\infty}  (Z)
\simeq
\int_0^\infty
e^{-s |\tr B|} \
\int_{\Omega}  e^{R(D_{-s}\, \tilde x)}\,
 {|\tilde x|}
\, dS (\tilde x)\,ds
\lesssim
\sqrt \beta \, e^{\beta} \int_0^\infty e^{-s |\tr B|}
\int_{\Omega}
 e^{-(R(\tilde x) - R(D_{-s}\, \tilde x))}\, dS (\tilde x)\,ds.
\end{align*}

By  \eqref{vel-4} and \eqref{est:2-eBs-v} we have
 \begin{align*}
 R(\tilde x) -R ({D_{-s}\, \tilde x})
\simeq
\int_0^s
\big|
D_{-s'} \,\tilde x\big|^2 \,ds'
\gtrsim
\int_0^s e^{-Cs'}\,|\tilde x|^2 \,ds'
\simeq  (1- e^{-Cs}) \, |\tilde x|^2 \simeq (s\wedge 1)\, \beta,
\end{align*}
which implies
\begin{align*}   
\gamma_{-\infty}  (Z)
\lesssim
\sqrt \beta \,
e^\beta\,  S(\Omega)\,
 \left[
\int_0^1  e^{-cs\beta} \, ds
 + e^{-c\beta} \,\int_1^\infty
  e^{-s |\tr B|}
\, ds \right].
 \end{align*}
Since $S(\Omega)\lesssim  a^{n-1}$ and both terms in the bracket are bounded by  $C/\beta$, this proves the lemma.
\end{proof}

\subsection{Decomposing the global region. }
In the following two sections,  we will  use a decomposition of  the global region into annuli. More precisely,  for  $t>0$ and $m = 1,2,\dots$  one  sets
 \begin{equation}\label{def:Tmt}
 \mathcal T^m_t:=
\left\{
(x,u)\in G\colon
2^{m-1}\big(1\wedge \sqrt t\big)<|u- D_{-t}\,x|
\le 2^{m} \big(1\wedge \sqrt t\big)
\right\}.
\end{equation}
But if $m=0$, we have only the upper estimate, setting
\begin{equation}\label{def:Tmt0}
 \mathcal T^0_t:=
 \left\{
(x,u)\in G \colon |u- D_{-t}\,x|
\le 1\wedge \sqrt t
\right\}.
\end{equation}
Note that for any fixed $t>0$ these sets form a partition of
$G$.

\vskip11pt

 \section{The global case for small $t$}\label{s:The global case}

\begin{proposition}\label{stima-tipo-debole-misto}
The maximal operator                               
$ \mathcal H_*^{-,G}$
 is of weak type $(1,1)$ with respect to the measure
 $d\gamma_{-\infty}$.
 \end{proposition}

We will
 prove this result in a way that  follows the proof of~\cite[Proposition 8.1]{CCS2}, but several adjustments are necessary to pass from the Ornstein--Uhlenbeck
framework to the inverse setting.
Here  $0<t \le 1$, and
as before, we let the function $f$ be nonnegative and normalized in  $L^1(\gamma_{-\infty})$.
It is enough to consider the level sets $\{ \mathcal H_*^{-,G} f > \alpha\}$ for
 $ \alpha < \alpha_0 $ with some small $\alpha_0 $; see the end of Subsection~\ref{simpleupper}.
Further, we need only consider points $x$ in the annulus $\mathcal E_\alpha$ defined in \eqref{def:annulus}.

What we must prove is that for all $0< \alpha < \alpha_0$
\begin{equation}\label{eq:obiettivo-finale}
\gamma_{-\infty}
\left\{
x\in{\mathcal E_\alpha}\colon
\sup_{0<t\le 1}
 \int
K_t^{UO}
(x,u)\, f(u)\,
d\gamma_{-\infty}(u)
\!
>\alpha \right\}
 \lesssim
\frac{1}{\alpha}.
\end{equation}

The global region $G$ is covered  by the sets $ \mathcal T^m_t, \;\,m\in \N$,  defined in~\eqref{def:Tmt} and~\eqref{def:Tmt0} and now given by
\begin{align}\label{def:Smt}
\mathcal T^{m}_t&=\left\{(x,u)\in G\colon 2^{m-1} \sqrt t<|u-D_{-t}\,x|\le 2^{m} \sqrt t\,\right\},
\end{align}
where the lower bound is to be suppressed for $m=0.$

From~\eqref{litet} we then see that for
$(x,u)\in {\mathcal T^{m}_t}$
\begin{align*} 
K_t^{UO}(x,u)
 &
\lesssim
\frac{
e^{-R(x)}
}{  t^{n/2}
}
\exp
\left({
-c{2^{2 m} }  }\right).
\end{align*}

Setting
\begin{align}\label{calKinTm}
{ K}_t^{-,{m}}
(x,u)
&=
\frac{
e^{-R(x)}
}{  t^{n/2}
}
\,\mathbf{1}_{\mathcal T^{m}_t}(x,u),
\end{align}
one has, for all $(x,u)\in G$ and $0<t \le 1$,
\begin{equation} \label{sumK}
K_t^{UO}(x,u)
\lesssim
\sum_{m=0}^\infty \exp
\left({
-c{2^{2 m} }  }\right)
{K}_t^{-,{m}}
(x,u)
\,.
\end{equation}
We define the operator
\begin{align*}
 \mathcal M_{m,\alpha}^{-}\, h(x) =  \mathbf{1}_{\mathcal E_\alpha} (x)\, \sup_{0<t\le 1} \int {K}_t^{-,{m}}
(x,u)\, |h(u)|\,
d\gamma_{-\infty}(u)
\end{align*}
 for $h \in L^1(\gamma_{-\infty})$,
and observe that~\eqref{sumK} implies
\begin{equation}\label{sumKK7}
\mathbf{1}_{\mathcal E_\alpha} (x)\,\sup_{0<t\le 1}
 \int
K_t^{UO}
(x,u)\, f(u)\,
d\gamma_{-\infty}(u)
\le \sum_{m=0}^\infty  \exp
\left({
-c{2^{2 m} }  }\right)\,\mathcal M_{m,\alpha}^{-}\, f(x).
\end{equation}
\begin{lemma}\label{sum}
 For $0< \alpha < \alpha_0$ with a suitably small $ \alpha_0 >0$  and any $m \in \N$, the operator $\mathcal M_{m,\alpha}^{-}$ maps $L^1(\gamma_{-\infty})$ into  $L^{1,\infty}(\gamma_{-\infty})$, with operator quasinorm at most $C\,2^{Cm}.$
\end{lemma}

Given this lemma,~\cite[Lemma 2.3]{steinweiss} will imply that the  $L^{1,\infty}(\mathcal E_\alpha; \gamma_{-\infty})$ quasinorm of the right-hand side of~\eqref{sumKK7} is bounded, uniformly in $ \alpha \in (0,  \alpha_0)$.
Then~\eqref{eq:obiettivo-finale} and Proposition~\ref{stima-tipo-debole-misto} will follow.

\begin{proof}[Proof of Lemma~\ref{sum}]                 
Let  $h\ge 0$ be  normalized in  $L^1(\gamma_{-\infty})$.
Fixing   $m\in \N$,  we must show that for any $\alpha' > 0$
\begin{equation}\label{eq:obiettivo'}
\gamma_{-\infty}
\left\{
x\in{\mathcal E_\alpha}:
\sup_{0<t\le 1}
 \int
{K}_t^{-,{m}}
(x,u)\, h(u)\,
d\gamma_{-\infty}(u)
\!
>\alpha' \right\}
 \lesssim
\frac{2^{Cm}}{\alpha'}.
\end{equation}
If
 $(x,u) \in \mathcal T_t^m$ for some $t \in (0,1]$, then
$|u-D_{-t}\, x |
\lesssim 2^{m}\sqrt{t}$, and~\eqref{forglobal}
implies
\begin{align*}
1& \lesssim
2^{m}\sqrt{t}\,(1+|x|)
+t|x|(1+|x|)
\le
2^{m}\sqrt{t}\,(1+|x|) + \left(2^{m}\sqrt{t}\,(1+|x|)\right)^2.
 \end{align*}
 It follows that      
\begin{equation}\label{stima-t-}
2^{m}\sqrt{t}\,(1+|x|) \gtrsim 1,
\end{equation}
and thus $t\gtrsim 2^{-2m} (1+|x|)^{-2}$. If also  $x \in \mathcal E_\alpha$ so that  $|x| \simeq \sqrt{\log(1/\alpha)}$, we conclude that
 $t\ge \delta>0$ for some
$\delta=\delta(\alpha,m)>0$. Hence,~\eqref{eq:obiettivo'} can be replaced by
\begin{equation}\label{eq:obiettivo-finale-new}
\gamma_{-\infty}
\left(\mathcal A_1 (\alpha')
 \right)
 \lesssim
\frac{2^{Cm}}{\alpha'}
,
\end{equation}
where
     \begin{equation*}
  \mathcal A_1 (\alpha')=
\left\{x\in {\mathcal E_\alpha}\colon
\sup_{\delta\le t\le 1}
\int
{ K}_t^{-,{m}}   (x,u)\,h(u)\,
d\gamma_{-\infty}(u)
\ge \alpha'
\right\}.
\end{equation*}
A benefit with this is that  the  supremum   is now a continuous function
of $x\in  {\mathcal E_\alpha}$, and the set $\mathcal A_1 (\alpha')$ is compact.

\medskip

Using the method from~\cite[Proposition 8.1]{CCS2},                         
we shall
prove~\eqref{eq:obiettivo-finale-new}
by
building a finite sequence of pairwise disjoint balls
$\big(\mathcal B^{(\ell)}\big)_{\ell=1}^{\ell_0}$ in $\R^n$ and at the same time a
finite sequence of tubes $\big(\mathcal Z^{(\ell)}\big)_{\ell=1}^{\ell_0}$ covering $ \mathcal A_1 (\alpha')$ and  called forbidden zones.

We will then verify the following three items:
\begin{enumerate}
\item
\hskip100pt the $ \mathcal B^{(\ell)}$ are pairwise disjoint;
\item   \begin{equation*}           
 \mathcal A_1 (\alpha') \subset \bigcup_{\ell=1}^{\ell_0}\mathcal Z^{(\ell)};
\end{equation*}
\item
 for each $\ell$
\begin{align*}                     
&\gamma_{-\infty} (\mathcal Z^{(\ell)})
\lesssim
\frac{2^{Cm}}{\alpha'}
 \int_{\mathcal B^{(\ell)}}h(u)
\,d\gamma_{-\infty}(u).
\end{align*}
\end{enumerate}
 This would imply~\eqref{eq:obiettivo-finale-new}
and Lemma \ref{sum}, as follows:
\begin{align*} 
\gamma_{-\infty}\left(\mathcal A_1 (\alpha') \right) \le
\gamma_{-\infty} \left(\bigcup_{\ell=1}^{\ell_0} \mathcal Z^{(\ell)}  \right)
 \lesssim
\frac{2^{Cm}}{\alpha'}\,
 \sum_{\ell=1}^{\ell_0}
  \int_{\mathcal B^{(\ell)}}h(u)\,d\gamma_{-\infty}(u)
 \lesssim
\frac{2^{Cm}}{\alpha'}.
\end{align*}

 \medskip

To construct the  sets $\mathcal B^{(\ell)}$ and $\mathcal Z^{(\ell)}$, we  define by recursion a sequence of points $x^{(\ell)}$,\; $\ell=1,\ldots, \ell_0$.
Let $x^{(1)}$ be
a maximum point
for the quadratic form $R(x)$   in the compact set  $ \mathcal A_1 (\alpha')$.
Notice that if this set is empty,~\eqref{eq:obiettivo-finale-new}
is immediate. Then by continuity  we
choose
 $t_1 \in [\delta, 1] $ such that
\begin{equation}\label{pha1}
 \int { K}_{t_1}^{-,{m}}   (x^{(1)},u)\,
h(u)\, d\gamma_{-\infty}(u) \ge \alpha'.
\end{equation}
Using this  $t_1$, we  associate with   $x^{(1)}$ the tube
\begin{equation*}
\mathcal Z^{(1)} =
\left\{
D_{ -s}\,\eta
\in \R^n
\colon s\ge 0,\;\;
R{( \eta
)}= R(x^{(1)}),
\;\;
| \eta- x^{(1)}
|<
A\, 2^{3m}\, \sqrt{t_{1}}\right\}.
\end{equation*}
The positive constant $A$ here
will depend  only on
$n$,
$Q$ and $B$ and will be  determined later.

\medskip

Recursively, suppose $x^{(\ell')}$, \hskip2pt $t_{\ell'}$ and $\mathcal Z^{(\ell')}$ have been defined for all $\ell'\leq \ell$, where $\ell \ge 1$. We choose   $x^{(\ell+1)}$
as a maximizing point of $R(x)$ in the set
 \begin{equation}
\label{def:set}
  \mathcal A_{\ell+1} (\alpha') :=
\bigg\{x\in {\mathcal E_\alpha}
  \setminus   \bigcup_{\ell'=1}^{\ell} \mathcal Z^{(\ell')}\colon 
\sup_{\delta\le t\le 1}
\int { K}_t^{-,{m}}   (x,u)\,h(u)\,d\gamma_{-\infty}(u)
\ge \alpha'
\bigg\},
\end{equation}
which is compact as shown later, provided this set is nonempty.  But if  $\mathcal A_{\ell+1} (\alpha')$
is empty,
the process stops with $\ell_0=\ell$ and item (2) follows.              
We will  see that this actually occurs for some  finite $\ell$.

Assume now that   $\mathcal A_{\ell+1} (\alpha') \ne \emptyset$.    By continuity, we can select
a
 $t_{\ell+1} \in [\delta, 1] $ such that
\begin{equation}\label{pha}
 \int {K}_{t_{\ell+1}}^{-,{m}}   \left(x^{(\ell+1)},u\right)\,
h(u)\, d\gamma_{-\infty}(u) \ge \alpha'.
\end{equation}

   Then we define the tube
\begin{equation*}
\mathcal Z^{(\ell+1)} =
\left\{
D_{- s}\,\eta
\in \R^n
\colon s\ge 0,\;\;
R{( \eta
)}= R\left(x^{(\ell+1)}\right),
\;\;
\left| \eta- x^{(\ell+1)}
\right|<
A\, 2^{3m}\, {\sqrt{t_{\ell+1}}}\right\}.
\end{equation*}

\medskip

It must be proved that
 $\mathcal A_{\ell+1} (\alpha')$
is closed and thus compact, even though the $\mathcal  Z^{(\ell')}$ are not open.
This will guarantee the existence of
a maximizing point.
We use induction, and observe that $\mathcal A_{1} (\alpha')$ is closed. Assume that
$\mathcal A_{\ell'} (\alpha')$ is closed for
  $1\le  \ell' \le \ell$.  For these $\ell'$,
 the maximizing property of $x^{(\ell')}$ shows that there is no point  $x$ in
$\mathcal A_{\ell'} (\alpha') $ with $R(x) > R(x^{(\ell')})$. Hence,
\begin{equation*}
  \mathcal A_{\ell+1} (\alpha') \subset  \mathcal A_{\ell'} (\alpha')
\subset \left\{x: R(x)\le R(x^{(\ell')})\right\}, \qquad 1\le  \ell' \le \ell,
\end{equation*}
and so
\begin{align}
   \mathcal A_{\ell+1} (\alpha')   
&= \bigcap_{1\le  \ell' \le\ell} \left(\mathcal A_{\ell+1} (\alpha')\,\cap\,         \label{cap}
\left\{x\colon R(x)\le R(x^{(\ell')})\right\}\right) \\
&=\bigcap_{1\le  \ell' \le\ell}
 \left\{x\in {\mathcal E_\alpha}
 \setminus   \mathcal Z^{(\ell')}\colon R(x)\le R(x^{(\ell')}), \;    
\sup_{\delta\le t\le 1}
\int {K}_t^{-,{m}}   (x,u)\,h(u)\,d\gamma_{-\infty}(u)          \notag
\ge \alpha'\right\}.
\end{align}

For each $\ell' = 1,\dots, \ell$ one has
\begin{multline*}
\{x\in {\mathcal E_\alpha} \setminus   \mathcal Z^{(\ell')}:
R(x)\le R(x^{(\ell')})\} \\
= \left\{D_{-s}\,\eta \in \mathcal E_\alpha \colon s\ge 0, \,R(\eta) = R(x^{(\ell')}), \;|\eta - x^{(\ell')}| \ge A\,2^{3m}\, \sqrt{t_{\ell'}} \right\},
\end{multline*}
and this set is closed. Now~\eqref{cap} shows that
  $  \mathcal A_{\ell+1} (\alpha')$ is closed,
  a maximizing point
$x^{(\ell+1)}$ can be chosen, and  the recursion is well defined.

By applying~\eqref{stima-t-}
to $t_\ell$ and
$x^{(\ell)}$,
one obtains, since  $|x^{(\ell)}|$ is large,
\begin{equation}\label{stima-t-prop}
 |x^{(\ell)}|^2\,
  2^{2m}\, t_\ell  \gtrsim 1.
\end{equation}
Then set
\begin{align*}
\mathcal B^{(\ell)}=&
\left\{
u
\in \R^n
\colon 
|u-
D_{-t_\ell}\, x^{(\ell)}
 \,
 |
 \le 2^{m} \sqrt{t_\ell}\,
\right\},\, \qquad 1\le  \ell' \le \ell.
\end{align*}
Combining~\eqref{calKinTm} and~\eqref{pha},   with  $\ell+1$ replaced by    $\ell$,                                
 we see   that
\begin{align}\label{mixed-bound-ell-1}
\alpha'&\le
\frac
{\exp \left(-{
R(x^{(\ell)})   }\right)
}{  t_\ell^{n/2}}
\int_{\mathcal B^{(\ell)}
}h(u)\,d\gamma_{-\infty}
 (u).
\end{align}
We will verify (1), (2), (3)
and start with (1).

The balls
${\mathcal  B}^{(\ell)}$ and ${\mathcal  B}^{(\ell')}$, with $\ell'<\ell$,
will be disjoint if
 \begin{equation}\label{stima-distanza-tang}
\big|
D_{-t_{\ell'}} \,
 x^{(\ell')}
 -D_{-t_{\ell}} \,
 x^{(\ell)}
 \big|
>
 2^m \left(\sqrt {t_\ell}+
 \sqrt{ t_{\ell'}}\,\right).
\end{equation}
By means of  our polar coordinates
with $\beta=R(x^{(\ell')})$,
we write
\begin{equation*}
x^{(\ell)}
=
D_{s}\, \tilde x^{(\ell)}
\end{equation*}
for some $\tilde x^{(\ell)}$ with
$R(\tilde x^{(\ell)})=
R(x^{(\ell')})$ and some $s \in\R$.
Note that  $s\le 0$,
since $R(x^{(\ell)})\le R( x^{(\ell')})$.
The point $x^{(\ell)}$ cannot belong to the forbidden zone $ \mathcal Z^{(\ell')}$,
so
\begin{equation}\label{hypo1}
|
\tilde x^{(\ell)}-
x^{(\ell')}
|\ge A\,
 2^{3m}\, \sqrt {t_{\ell'}}\,.
 \end{equation}

 Since $ t_{\ell'} \le 1$,~\eqref{est:2-eBs-v} implies $R(D_{-t_{\ell'}} \,x^{(\ell')}) > c\, \beta$.
  This allows us to  apply~\eqref{lem1} and the observation following it, to obtain
\begin{equation} \label{distan}
\big|
D_{-t_{\ell'}} \,
 x^{(\ell')}
 -D_{-t_{\ell}} \,
 x^{(\ell)}
 \big| \gtrsim  A \,2^{3m} \,\sqrt {t_{\ell'}}.
\end{equation}
If $ \sqrt A\, 2^{2m} \sqrt {t_{\ell'}} \ge \sqrt {t_{\ell}}\,$\,,
we will  have
\begin{equation}\label{firstcase}
A 2^{3m} \sqrt {t_{\ell'}} \ge \frac12 A 2^{3m} \sqrt {t_{\ell'}} +\frac{1}{2} \sqrt A\, 2^{m} \sqrt {t_{\ell}}\,,
\end{equation}
 and~\eqref{stima-distanza-tang} follows from~\eqref{distan}, provided  $A$ is large enough.

 It remains to make the contrary assumption  $ \sqrt {t_{\ell}} > \sqrt A\, 2^{2m} \sqrt {t_{\ell'}}$,
 which implies in particular that  $t_{\ell} > t_{\ell'}$.
Observe that
\begin{align} \label{dist}
\big|
D_{-t_{\ell'}} \,&
 x^{(\ell')}-
D_{-t_{\ell}} \,
 x^{(\ell)}\big|
=
\big|D_{-t_{\ell}}\big(
D_{t_{\ell} -t_{\ell'}} \, x^{(\ell')}-
 x^{(\ell)}\big)\big|
\simeq
\big|
D_{t_{\ell} -t_{\ell'}} \, x^{(\ell')}-
D_s\,\tilde  x^{(\ell)}\big|.
\end{align}
Both $x^{(\ell)}$ and  $x^{(\ell')}$ are  in $\mathcal E_\alpha$, so they satisfy
\[
R(x^{(\ell)}) \geq \frac34 \,\log \frac{1}{\alpha} \geq \frac{3/4}{5/4}\, R(x^{(\ell')})>\frac{1}{2} \, R(x^{(\ell')})=\frac12\,\beta\,.
\]
Since also  $t_{\ell}- t_{\ell'}\geq 0$, we can apply~\eqref{lem2} to the last expression in~\eqref{dist}
and get
\begin{align}\label{distance}
\big|
D_{-t_{\ell'}} \,
 x^{(\ell')}-
D_{-t_{\ell}} \,
 x^{(\ell)}\big|
 &\gtrsim   |t_{\ell}-t_{\ell'}-s|
\big|x^{(\ell')}\big|\,
\gtrsim  { t_{\ell}}\,
\big|x^{(\ell')}\big|\,,
\end{align}
 the second step because  our assumption implies $t_{\ell}-t_{\ell'} \simeq t_{\ell}$, and $s<0$.
By means again of this  assumption and then~\eqref{stima-t-prop}, we find
\begin{equation*}              
 { t_{\ell}}\, \big|x^{(\ell')}\big|\gtrsim  \sqrt{ t_{\ell}}\,
\big|x^{(\ell')}\big|\, \sqrt{A} \, 2^{2m} \sqrt{ t_{\ell'}}
\gtrsim \sqrt{A} \, 2^{m}\sqrt{ t_{\ell}}
\simeq  \sqrt{A}\, 2^{m}\,\left( \sqrt{ t_{\ell}} +\sqrt{ t_{\ell'}\,}\,\right).
\end{equation*}
With $A$  large enough,~\eqref{stima-distanza-tang} now follows from this and~\eqref{distance}. Item (1) is verified.

Next, we will prove item (2).
For $\ell'<\ell$, we can apply~\eqref{lem1} and then~\eqref{hypo1}, 
to  get
\begin{align*}
\big|
x^{(\ell')}-
 x^{(\ell)}
\big|&
\gtrsim
\big|
x^{(\ell')}-
 \tilde x^{(\ell)}
\big|
\gtrsim
 A\,
 2^{3m} \sqrt {t_{\ell'}}\,.
\end{align*}
Since
$t_{\ell'}\ge \delta$,
 the distances
$\left|
x^{(\ell')}-
x^{(\ell)}
\right|$
are thus bounded below by a positive constant.
This implies that
the sequence $(x^{(\ell)})$
is finite,
since all the $
x^{(\ell)}$
are contained in the bounded set
$ {\mathcal E_\alpha}$.
Thus the set  $  \mathcal A_{\ell+1} (\alpha)$ defined in~\eqref{def:set}
will be  empty for some $\ell$, say $\ell = \ell_0$, and the recursion stops.  This implies
 item (2).

\medskip

We are left with the proof of item (3).
The  forbidden zone
$\mathcal Z^{(\ell)}$
is a tube  as  defined in~\eqref{zona}, with
 $a=A\, 2^{3m}
 \sqrt{t_\ell}$ and
 $\beta=R(x^{(\ell)})$.
This value of  $\beta$ will be large since $x^{(\ell)} \in {\mathcal E_\alpha}$ and $\alpha < \alpha_0$ for some small $ \alpha_0$. Thus we can apply
Lemma~\ref{lemma-Peter-forbidden}
to  obtain
\begin{align*}
\gamma_{-\infty} (  \mathcal Z^{(\ell)})\lesssim
\frac{\big(  A 2^{3m} \sqrt{t_\ell}\big)^{n-1}}{
\sqrt{ R(x^{(\ell)})}}\, \exp\left(R\left(x^{(\ell)}\right)   \right)
\notag
. \end{align*}
We bound the exponential here by  means of~\eqref{mixed-bound-ell-1}
and observe that $R(x^{(\ell)}) \simeq |x^{(\ell)}|^2$, getting
\begin{align*}
\gamma_{-\infty} \left(\mathcal Z^{(\ell)}\right)
&\lesssim \frac{1}{\alpha'\, |x^{(\ell)}| \,
{\sqrt{t_\ell}}} \,
 (A 2^{3m})^{n-1} \,
\int_{\mathcal B^{(\ell)}
}\!h(u)\,d\gamma_{-\infty} (u)\,.
\end{align*}
As a consequence of this and
\eqref{stima-t-prop},
we obtain
\begin{align*}
\gamma_{-\infty}
 (\mathcal Z^{(\ell)})
\lesssim
\frac{2^{m}}{\alpha'} \,
 \big(A 2^{3m}\big)^{n-1} \,
\int_{\mathcal B^{(\ell)}}h(u)\, d\gamma_{-\infty}
 (u)\,
\lesssim
\frac{2^{Cm}}{\alpha'}\,
 \, \int_{\mathcal B^{(\ell)}}h(u)\, d\gamma_{-\infty}
 (u)\,,
\end{align*}
  which proves
item    (3).  This completes the proof of  Lemma \ref{sum} and that of Proposition~\ref{stima-tipo-debole-misto}.      
\end{proof}

\section{The global case for large $t$}\label{s:mixed. t large}

This section consists of the proof of the following result.

\begin{proposition}\label{stima-tipo-debole-misto_tlarge}
For all functions $f\in L^1(\gamma_{-\infty})$ with  $\|f\|_{L^1(\gamma_{-\infty})}=1$, 
\begin{equation}\label{eq:proptlarge}
\gamma_{-\infty}
\left\{x\colon  \left|\mathcal H_*^{+,G}
f(x)\right|
>\alpha \right\}
 \lesssim
\frac{1}{\alpha\,\sqrt{\log (1/\alpha)}}\,,\qquad\alpha\in (0,1/2).
\end{equation}
In particular,
the maximal operator
 $ \mathcal H_*^{+,G}$ maps $L^1(\gamma_{-\infty})$ into  $L^{1,\infty}(\gamma_{-\infty})$.
 \end{proposition}

Observe first that the second statement follows from the first together with the observation at the end of Subsection~\ref{simpleupper}. The same observation allows us to reduce the range of $\alpha$ in the first statement to $\alpha < e^{-2}$, in the proof that follows.

The proof of the first statement runs at first like that of Proposition~\ref{stima-tipo-debole-misto}, although now $t>1$.                 
In particular,   $f \ge 0$ is normalized in  $ L^1(\gamma_{-\infty})$, and from Subsection~\ref{smalllevels} we know that it is enough to consider the values of $\mathcal H_*^{+,G}
f$ in the set $\mathcal E_{\alpha}$ defined in~\eqref{def:annulus}. The annuli $\mathcal T^{m}_t$ are now
\begin{align*}
\mathcal T^{m}_t
&=\left\{(x,u)
\in G\colon 
2^{m-1}\le |u- D_{-t}\,x|< 2^{m}
\,\right\}
, \qquad m \in \N,
\end{align*}
without the lower bound when $m=0.$

We must show that for $0< \alpha < e^{-2}$
\begin{equation}\label{eq:obiettivo-8}
\gamma_{-\infty}
\left\{
x\in{\mathcal E_\alpha}\colon
\sup_{t > 1}
 \int
K_t^{UO}
(x,u)\, f(u)\,
d\gamma_{-\infty}(u)
\!
>\alpha \right\}
 \lesssim
\frac{1}{\alpha\,\sqrt{\log (1/\alpha)}}\,.
\end{equation}

Setting $T = | \text{tr} B| > 0$, we obtain from~\eqref{stort} that for any
 $(x,u)\in{\mathcal T^{m}_t}$
\begin{align*}
K_t^{UO}(x,u) &\lesssim
 e^{-T \,t}\,  e^{-R(x)}
\exp
\left({
-c{2^{2 m} }  }\right)\,.
\end{align*}
Observing that  $(x,u)\in{\mathcal T^{m}_t}$ implies  $u \in B(D_{-t}\,x,\, 2^{m})$, we set
$$
{ K}_t^{+,{m}}(x,u)   = e^{-T \,t}\,  e^{-R(x)}\, \mathbf{1}_{B(D_{-t}\,x,\, 2^{m})}(u)\,,
$$
and conclude that for any
 $(x,u)\in G$ and $t>1$
\begin{equation*}
K_t^{UO}(x,u)
   \lesssim
\sum_{m=0}^\infty  \exp
\left({
-c{2^{2 m} }  }\right) { K}_t^{+,{m}}
(x,u)\,.
\end{equation*}
Almost as in the preceding section, we introduce the operators
\begin{align*}
 \mathcal M_{m, \alpha}^{+}\, h(x) = \mathbf{1}_{\mathcal E_\alpha}(x)\,  \sup_{t > 1} \int {K}_t^{+,{m}}
(x,u)\, |h(u)|\,
d\gamma_{-\infty}(u)\,,
\end{align*}
so that 
\begin{equation}\label{sumKK}  \mathbf{1}_{\mathcal E_\alpha}(x)\,
\sup_{t > 1}
 \int
K_t^{UO}
(x,u)\, f(u)\,
d\gamma_{-\infty}(u)
\lesssim \sum_{m=0}^\infty  \exp
\left({
-c{2^{2 m} }  }\right)\,\mathcal M_{m,\alpha}^{+} f(x).
\end{equation}

\begin{lemma}\label{sum8}
 For $0< \alpha < e^{-2}$ and $m \in \N$, the operator $\mathcal M_{m,\alpha}^{+}$ maps $L^1(\gamma_{-\infty})$ into  $L^{1,\infty}(\gamma_{-\infty})$, with operator quasinorm at most
 $C\,2^{Cm}/{\sqrt{\log (1/\alpha)}\,}.$
\end{lemma}

In analogy with Lemma~\ref{sum}, this lemma implies~\eqref{eq:obiettivo-8} and thus also Proposition~\ref{stima-tipo-debole-misto_tlarge}.

\begin{proof}
With $m$ and  $\alpha $ fixed, we will estimate $\mathcal M_{m,\alpha}^{+}\, h$ for a function $h\ge 0$ which is normalized in  $L^1(\gamma_{-\infty})$.  But we prefer to work with the function $g(u) = (2\pi)^{\frac{n}{2}}
(\text{det} \, Q_\infty)^{\frac{1}{2} }\,e^{R(u)} \,h(u)$, which  is normalized in $ L^1(\R^n, du)$, and nonnegative. Then
\begin{align*}
 \mathcal M_{m,\alpha}^{+} \, h(x)
& = \mathbf{1}_{\mathcal E_\alpha}(x)\, \sup_{t>1}\, e^{-Tt}\, e^{-R(x)}
\int_{B(D_{-t}\,x,\, 2^{m})}
h(u)\, d\gamma_{-\infty}(u)\\
&= \mathbf{1}_{\mathcal E_\alpha}(x)\,
\sup_{t>1}\, e^{-Tt} \,e^{-R(x)}
\int_{B(D_{-t}\,x,\, 2^{m})}
g(u)\,du.
\end{align*}

With $r>0$ we write $g_r = g*\mathbf{1}_{B(0,r)}$, which is for each $r>0$ a continuous function in  $L^\infty\cap L^1(\R^n, du)$. Then
\begin{align*}
\mathcal M_{m,\alpha}^{+} \, h(x)
  &= \mathbf{1}_{\mathcal E_\alpha}(x)\,
\sup_{t>1}\,
 e^{-Tt}\, e^{-R(x)}\, g_{2^m}(D_{-t}\,x),
\end{align*}
 and as a supremum of continuous functions, $M_{m,\alpha}^{+} \, h$ is lower semicontinuous when restricted to $\mathcal E_\alpha$.

Let $\alpha' > 0$. Clearly, $\mathcal M_{m,\alpha}^{+} \, h(x) > \alpha'$ if and only if  $x\in{\mathcal E_\alpha}$ and there exists a $t>1$ such that $e^{-R(x)} \,e^{-Tt}\, g_{2^m}(D_{-t}\,x)>\alpha'$.
We use polar coordinates, writing points $x\in \mathcal E_{\alpha}$ as $x = D_\varrho\,\tilde x$, where
$\varrho \in \R$ and $\tilde x$ is on the ellipsoid $ E_1 =\{y:R(y) = 1\}.$
Notice that actually $\varrho > 0$ here, since $\alpha < e^{-2}$ implies that $R(x) > 1$ for   $x\in \mathcal E_{\alpha}$.

Let $ A_{\alpha'}$ be the set of points $\tilde x \in  E_1$ for which there exists a
 $\varrho > 0$ such that  $D_\varrho\, \tilde x \in \mathcal E_\alpha$ and
  $\mathcal M_{m,\alpha}^{+} \, f(D_\varrho \,\tilde x)>\alpha'$.
The lower semicontinuity of  $\mathcal M_{m,\alpha}^{+} \, f$ shows that  $ A_{\alpha'}$ is a relatively open subset of  $E_1$.
For  $\tilde x \in A_{\alpha'}$ we define
 \begin{align*}
\varrho(\tilde x)
\!&= \sup\,\{\varrho>0 \colon \mathcal M_{m,\alpha}^{+} \,h(D_{\varrho}\,\tilde x)>\alpha'\}\\
&= \sup\, \{\varrho > 0\colon \exists\, t>1 \quad \mathrm{with}\quad \mathbf{1}_{\mathcal E_\alpha}(D_\varrho\,\tilde x)\, e^{-Tt}\, e^{-R(D_\varrho\,\tilde x)}\, g_{2^m}(D_{-t}\,D_\varrho\,\tilde x)>\alpha'\}.
\end{align*}
For notational convenience, we do not indicate that $\varrho(\tilde x)$ also depends on $\alpha'.$

We claim  that $\varrho(\tilde x)$ is a lower semicontinuous function on $ A_{\alpha'}\,,$ and thus measurable. Indeed, if
 $\varrho(\tilde x) > \varrho_0$ for some $\tilde x \in A_{\alpha'}$ and some $\varrho_0 >0$, then there exists a
 $\varrho' > \varrho_0$ such that $\mathcal M_{m,\alpha}^{+} \,h(D_{\varrho'}\,\tilde x)>\alpha'$. By the  lower semicontinuity of  $\mathcal M_{m,\alpha}^{+} \, h$, the same inequality holds at  $D_{\varrho'}\,\tilde x'$ for points $\tilde x'  \in A_{\alpha'}$ in a small neighborhood of  $\tilde x$. The claim follows.

For each  $\tilde x \in A_{\alpha'}$
 we can choose a sequence
 $\varrho_{i} \nearrow \varrho(\tilde x)$
  and another sequence  $(t_{i})$ in $(1,\infty)$
 satisfying
\begin{equation}\label{eq:measurability}
 e^{-T{t_{i}}}\, e^{-R(D_{\varrho_{i}}\,\tilde x)}\, g_{2^m}(D_{-t_{i}}\,D_{\varrho_{i}}\,\tilde x)>\alpha'\,.
\end{equation}
This inequality implies that
$e^{R(D_{\varrho_{i}}\,\tilde x)} < \|g\|_{L^\infty}/\alpha'$ and $e^{Tt_{i}} < \|g\|_{L^\infty}/\alpha'$,
from which we conclude that $\varrho(\tilde x) < \infty$ and  that the  $t_{i}$ stay bounded.
  A suitable subsequence of  $(t_{i})$ will then converge, say to $t(\tilde x) \ge 1$.

From~\eqref{eq:measurability} in the limit we get
\begin{align} \label{limit}
e^{-R(D_{\varrho(\tilde x)}\,\tilde x)}\, e^{-T{t(\tilde x)}}\, g_{2^m}(D_{-t(\tilde x)}\,D_{\varrho(\tilde x)}\,\tilde x)\ge \alpha'.
\end{align}
The level set $\{x: \mathcal M_{m,\alpha}^{+}\,h(x) > \alpha'\}$ is contained in the set
\[
\mathcal F_{\alpha'} = \{D_{\varrho }\,\tilde x \colon \; \tilde x \in A_{\alpha'}\,,\;\; 0 < \varrho \le \varrho(\tilde x) \}\,.
\]
To estimate the measure of this set, we switch to polar coordinates with $\beta = 1$ and use~\eqref{def:leb-meas-pulita}; thus
  \begin{equation}\label{ante-lemma}
 \gamma_{-\infty} (\mathcal F_{\alpha'}) = \int_{\mathcal F_{\alpha'}} e^{R(x)}\,dx \lesssim  \int_{A_{\alpha'}} \int_{0}^{\varrho(\tilde x)} e^{T\varrho} \,e^{R(D_{\varrho }\,\tilde x)}\,d\varrho\, dS(\tilde x)\,.
   \end{equation}
Here $dS(\tilde x)$ is the area measure on $E_1$. We now estimate the inner integral above.

\begin{lemma} \label{intdrho}
One has for  $\tilde x \in A_{\alpha'}$
  \begin{align*}
\int_{0}^{\varrho(\tilde x)} e^{T\varrho}\, e^{R(D_{\varrho }\,\tilde x)}\,d\varrho \lesssim
 e^{T\varrho(\tilde x)}\, e^{R(D_{\varrho(\tilde x) }\,\tilde x)}\,|D_{\varrho(\tilde x) }\,\tilde x|^{-2}\,.
  \end{align*}

\end{lemma}

\begin{proof}
  When proving this, we can delete the factors $e^{T\varrho}$ and $e^{T\varrho(\tilde x)}$.

  If $0<\varrho < \varrho(\tilde x) - A$ for some  $A>0$,~\eqref{est:2-eBs-v}
  implies
  \begin{equation*}
    R(D_\varrho \,\tilde x) = R(D_{\varrho - \varrho(\tilde x)}\,D_{\varrho(\tilde x)}\, \tilde x) \lesssim
    e^{-c( \varrho(\tilde x)- \varrho)}\,R(D_{\varrho(\tilde x)}\, \tilde x) \le
    e^{-cA}\,R(D_{\varrho(\tilde x)}\, \tilde x)\,.
  \end{equation*}
  Choosing $A = A(n,Q,B)$ large enough, we conclude that
    \begin{equation*}
    R(D_\varrho \,\tilde x)  \le
    \frac12\,R(D_{\varrho(\tilde x)}\, \tilde x)\,, \qquad \varrho < \varrho(\tilde x) - A\,,
  \end{equation*}
  and thus
   \begin{align*}
\int_{0}^{\varrho(\tilde x) - A}  e^{R(D_{\varrho }\,\tilde x)}\,d\varrho \lesssim &\:
 \exp\left(\frac12\,R(D_{\varrho(\tilde x)}\, \tilde x) \right) \varrho(\tilde x)  \\
=&\, \exp\left(R(D_{\varrho(\tilde x)}\, \tilde x)\right)\,\left[\exp\left(-\frac14\,R(D_{\varrho(\tilde x)}\, \tilde x)\right)\right]^2\, \varrho(\tilde x)\,.
 \end{align*}
 Here one of the factors $\exp\left(-R(D_{\varrho(\tilde x)}\, \tilde x)/4\right)$ takes care of $\varrho(\tilde x)$ in view of~\eqref{est:2-eBs-v}, and the other is no larger than $C |D_{\varrho(\tilde x) }\,\tilde x|^{-2}$.
 Hence, this part of the integral in the lemma satisfies the desired estimate.

For   $\varrho(\tilde x) - A < \varrho < \varrho(\tilde x)$ we use~\eqref{vel-4}, the fact that
  the quantity $|D_s\,\tilde x|$ is increasing in $s$ and finally~\eqref{est:2-eBs-v}, to get
   \begin{align*}
R( D_{\varrho(\tilde x)}\, \tilde x)  -   R(D_\varrho \,\tilde x) 
 & \simeq \, \int_{\varrho}^{  \varrho(\tilde x) } |D_s\,\tilde x|^2\,ds
  \ge (\varrho(\tilde x) - \varrho)\, |D_\varrho\,\tilde x|^2 
   = (\varrho(\tilde x) - \varrho)\, |D_{\varrho - \varrho(\tilde x)}\,D_{\varrho(\tilde x)}\,\tilde x|^2 \\
 & \gtrsim  (\varrho(\tilde x) - \varrho)\, e^{-CA}\,|D_{\varrho(\tilde x)}\,\tilde x|^2
  \simeq (\varrho(\tilde x) - \varrho)\,|D_{\varrho(\tilde x)}\,\tilde x|^2\,.
        \end{align*}
  Thus we can write
   \begin{align*}
\int_{\varrho(\tilde x) - A}^{\varrho(\tilde x)} e^{R(D_{\varrho }\,\tilde x)}\,d\varrho 
  \lesssim &\: e^{R(D_{\varrho(\tilde x) }\,\tilde x)}\, \int_{\varrho(\tilde x) -A}^{  \varrho(\tilde x) }
  e^{-c(\varrho(\tilde x) - \varrho)\,|D_{\varrho(\tilde x) }\,\tilde x|^{2}\,}\,d\varrho \\
  = &\: e^{R(D_{\varrho(\tilde x) }\,\tilde x)}\, \int_{0}^{A} e^{-c\sigma\,|D_{\varrho(\tilde x) }\,\tilde x|^{2}\,}\,d\sigma  \lesssim
   e^{R(D_{\varrho(\tilde x) }\,\tilde x)}\,
   |D_{\varrho(\tilde x) }\,\tilde x|^{-2}\,.
  \end{align*}
 The proof of Lemma~\ref{intdrho} is complete.
\end{proof}

To continue the proof of Lemma~\ref{sum8}, we now insert in~\eqref{ante-lemma}  the expression from Lemma~\ref{intdrho}, and obtain
  \begin{equation}\label{ineq_h_E}
\gamma_{-\infty} (\mathcal F_{\alpha'}) \lesssim \int_{A_{\alpha'}}
 e^{T\varrho(\tilde x)}\, e^{R(D_{\varrho(\tilde x) }\,\tilde x)}\,|D_{\varrho(\tilde x) }\,\tilde x|^{-2} \, dS(\tilde x)\,.
   \end{equation}
Then we apply the upper estimate of $e^{R(D_{\varrho(\tilde x)}\,\tilde x )}$ that follows from~\eqref{limit}, to conclude that
\begin{equation}\label{extra}
e^{T\varrho(\tilde x)}\, e^{R(D_{\varrho(\tilde x) }\,\tilde x)}\,|D_{\varrho(\tilde x) }\,\tilde x|^{-2}  \lesssim \frac1{\alpha'} \, \psi(\tilde x)\,
\end{equation}
where
\[\psi(\tilde x)=
  e^{T\varrho(\tilde x)}\,   e^{-Tt(\tilde x)}\,      |D_{\varrho(\tilde x) }\,\tilde x|^{-2}\,
  g_{2^m}(D_{-t(\tilde x)}D_{\varrho(\tilde x)}\,\tilde x)\,.\]

Let $|\sigma| \le 1\wedge |D_{-t (\tilde x)}\,D_{\varrho (\tilde x)}\,\tilde x|^{-1}$. Then~\eqref{x-Dtx} leads to
\[
 |D_\sigma\, D_{-t (\tilde x)}\, D_{\varrho (\tilde x)}\,\tilde x - D_{-t (\tilde x)}\,D_{\varrho (\tilde x)}\,\tilde x| \lesssim
 |\sigma |\, | D_{-t (\tilde x)}\, D_{\varrho (\tilde x)}\,\tilde x| \leq 1,
\]
and the inclusion $B(D_{-t (\tilde x)}\, D_{\varrho (\tilde x)}\,\tilde x,2^m) \subset B(D_{\sigma}\,D_{-t(\tilde x)}\,D_{\varrho(\tilde x)}\,\tilde x,2^m+C)$ yields
\begin{equation*}           
g_{2^m}(D_{-t(\tilde x)}\,D_{\varrho(\tilde x)}\,\tilde x) \le g_{2^{m}+C}(D_{\sigma}\,D_{-t(\tilde x)}\,D_{\varrho(\tilde x)}\,\tilde x).
 \end{equation*}

Thus we can write for $\tilde x \in A_{\alpha'}$
 \begin{equation}
\psi(\tilde x)\le
 |D_{\varrho(\tilde x) }\,\tilde x|^{-2}\,e^{-T \sigma} \,e^{T (\sigma-t(\tilde x) +\varrho(\tilde x))}\,
   g_{2^{m}+C}(D_{\sigma -t(\tilde x)+\varrho(\tilde x)}\,\tilde x).
   \end{equation}
   Here we replace $e^{-T \sigma}$ by $C$ and then take
   the mean of both sides with respect to $\sigma$ in the indicated interval,
to get
    \begin{align*}
\psi(\tilde x)
& \lesssim 1\vee |D_{-t(\tilde x)}\,D_{\varrho(\tilde x)}\,\tilde x| \\&
\qquad \quad \times \int_{|\sigma| \le 1\wedge |D_{-t(\tilde x)}D_{\varrho(\tilde x)}\,\tilde x|^{-1}}
\, |D_{\varrho(\tilde x) }\,\tilde x|^{-2} \,e^{T (\sigma-t(\tilde x) +\varrho(\tilde x))}\, g_{2^{m}+C}(D_{\sigma -t(\tilde x)+\varrho(\tilde x)}\,\tilde x) \, d\sigma\\
& \lesssim          |D_{\varrho(\tilde x)}\,\tilde x|^{-1} \,
\int_{|\sigma| \le 1\wedge\, |D_{\varrho(\tilde x) }\,\tilde x|^{-1} }
e^{T (\sigma-t(\tilde x) +\varrho(\tilde x))}\,   g_{2^{m}+C}(D_{\sigma -t(\tilde x)+\varrho(\tilde x)}\,\tilde x) \, d\sigma,
   \end{align*}
    where we used the fact that  $\varrho(\tilde x) > 0$ and $\tilde x\in\mathcal E_1$ so that
    $1\vee |D_{-t(\tilde x)}\,D_{\varrho(\tilde x)}\,\tilde x| \le |D_{\varrho(\tilde x) }\,\tilde x|$.
Since $D_{\varrho(\tilde x)}\,\tilde x \in \mathcal{E_\alpha}$, we  have $|D_{\varrho(\tilde x)}\,\tilde x|^{-1} \lesssim 1/\sqrt{\log(1/\alpha)}$.
By replacing $\sigma$ by $s = \sigma-t(\tilde x) +\varrho(\tilde x)$ and extending the integral to all of $\R$, we  obtain
  \begin{align}
\psi(\tilde x ) \lesssim
\frac1{\sqrt{\log(1/\alpha)}} \, \int_{\R}
 e^{T s}\,
   g_{2^{m}+C}(D_{s}\,\tilde x)\, ds.
   \end{align}
  Inserting this  in~\eqref{ineq_h_E} combined with~\eqref{extra}, we find that
   \begin{align}
\gamma_{-\infty} (\mathcal F_{\alpha'})  \lesssim  \frac1{\alpha'} \,\frac1{\sqrt{\log(1/\alpha)}} \,\int_{A_\alpha'}
    \, \int_{\R}     e^{T s}\,   g_{2^{m}+C}(D_{s}\,\tilde x)\, ds\, dS(\tilde x).
   \end{align}
Now we use~\eqref{def:leb-meas-pulita}
  to go back to  Lebesgue measure $dx$ with $x =D_s\,\tilde x$.
Since  $|\tilde x| \simeq 1$,
this yields
  \begin{align*}
 \gamma_{-\infty} (\mathcal F_{\alpha'})
 &\lesssim \frac{1}{\alpha'} \,\frac1{\sqrt{ \log(1/\alpha)}}
 \int_{\{x = D_s\,\tilde x:\:s \in \R,\: \tilde x \in A_{\alpha'}\} }
   g_{2^{m}+C}   (x)\, dx\\
 &  \le \frac{1}{\alpha'} \,\frac1{  \sqrt{ \log(1/\alpha)}} \int_{\R^n} g_{2^{m}+C}(x)\, dx \\
 & \lesssim \frac{1}{\alpha'} \,\frac1{  \sqrt{\log(1/\alpha)}} \, \left(2^{m}+C\right)^n \, \|g\|_{L^1(\R^n)} \simeq \frac{1}{\alpha'} \frac{2^{mn}}{\sqrt{\log(1/\alpha)}}.
   \end{align*}
   This ends the proof of Lemma~\ref{sum8} and that of Proposition~\ref{stima-tipo-debole-misto_tlarge}.
\end{proof}

\section{Completion of proofs}\label{completion}
Theorem~\ref{weaktype1} is an immediate consequence of
Propositions~\ref{prop-locale},~\ref{stima-tipo-debole-misto} and~\ref{stima-tipo-debole-misto_tlarge}.
Further,    Propositions~\ref{prop-locale} and~\ref{stima-tipo-debole-misto_tlarge} together imply the positive part of
Theorem~\ref{c:sharp}. It remains for us to prove the sharpness assertion in Theorem~\ref{c:sharp}.

To this end, we take a  point $z$ with  $R(z)$ large.  Let $B_1$ denote the ball $ B(z,1)$ and set 
 $B_2 =  B(D_{-2 }\,z,1)$.  If $x \in B_1$ and $u \in B_2$, we will have 
 $$
 |D_{-2 }\, x - u| \le |D_{-2 }\,x - D_{-2 }\, z| + |D_{-2 }\, z - u| = |D_{-2 }\,( x - z)| + |D_{-2 }\, z - u| \lesssim 2,
 $$
  so that \eqref{stort} yields $ K^{UO}_2(x,u) \simeq e^{-R(x)}$.
With $f = \mathbf 1_{B_2}/\gamma_{-\infty}(B_2)$ it follows that $\mathcal H^{UO}_2 f(x) \simeq  e^{-R(x)} $ for $x \in B_1$.
 
Then  $\mathcal H^{UO}_2 f(x) \gtrsim  e^{-R(z)}$ if $x$ is in the set 
$B^* = \{x \in B_1: R(x) < R(z) \}$.
For any small $\alpha >0$ we can  choose $z$ satisfying  $R(z) = \log ( 1/ \alpha) -A$ for a suitable, large constant $A = A(n,Q,B)$, and conclude  that
$\mathcal H^{UO}_2 f(x) > \alpha$ for $x \in B^*$.  
With our polar coordinates, one can verify that
\begin{equation*}
  \gamma_{-\infty}(B^*) \gtrsim \frac{e^{R(z)}}{\sqrt{R(z)}}  \simeq \frac{1}{\alpha\,\sqrt{\log( 1/\alpha)} }
\end{equation*}
if  $\alpha $ is small enough.
This means that \eqref{enhanced} implies \eqref{PHI}.

 Theorem~\ref{c:sharp} is completely proved.

\end{document}